\documentclass[13pt]{article}
\usepackage[a4paper, margin = 3cm]{geometry}
\usepackage{amsmath, amssymb, bbm, tikz-cd, url}
\usepackage{amsthm}
\usepackage{thmtools}
\usepackage[hidelinks]{hyperref}
\usepackage{cleveref}
\usepackage[affil-it]{authblk}

\DeclareSymbolFont{bbold}{U}{bbold}{m}{n}
\DeclareSymbolFontAlphabet{\mathbbold}{bbold}

\usepackage{amsmath, amsthm, stmaryrd, amsfonts, cleveref, tikz-cd, enumitem}

\newcommand{\ind}{\leavevmode{\parindent=15pt\indent}}

\theoremstyle{plain} 
\newtheorem{theorem}{Theorem}[section]
\newtheorem*{theorem*}{Theorem}
\newtheorem{corollary}[theorem]{Corollary}
\newtheorem*{corollary*}{Corollary}

\newtheorem*{proposition*}{Proposition}
\newtheorem{lemma}[theorem]{Lemma}
\newtheorem*{lemma*}{Lemma}

\newtheorem*{fact*}{Fact}

\newtheorem*{conjecture*}{Conjecture}

\newtheorem*{criterion*}{Criterion}

\newtheorem*{assertion*}{Assertion}

\newtheorem*{lem_def*}{Lemma-Definition}

\newtheorem*{prop_def*}{Proposition_Definition}

\newtheorem*{thm_def*}{Theorem-Definition}

\theoremstyle{definition} 
\newtheorem{example}[theorem]{Example}
\newtheorem*{example*}{Example}

\newtheorem*{examples*}{Examples}
\newtheorem{definition}[theorem]{Definition}
\newtheorem*{definition*}{Definition}

\newtheorem*{condition*}{Condition}

\newtheorem*{problem*}{Problem}

\newtheorem*{exercise*}{Exercise}

\newtheorem*{algorithm*}{Algorithm}

\newtheorem*{subroutine*}{Subroutine}

\newtheorem*{question*}{Question}

\newtheorem*{axiom*}{Axiom}

\newtheorem*{property*}{Property}

\newtheorem*{assumption*}{Assumption}

\newtheorem*{hypothesis*}{Hypothesis}

\theoremstyle{remark} 
\newtheorem{remark}[theorem]{Remark}

\newtheorem*{remark*}{Remark}

\newtheorem*{note*}{Note}

\newtheorem*{scholium*}{Scholium}

\newtheorem*{notation*}{Notation}

\newtheorem*{claim*}{Claim}

\newtheorem*{summary*}{Summary}

\newtheorem*{acknowledgment*}{acknowledgment}

\newtheorem*{acknowledgement*}{acknowledgement}

\newtheorem*{case*}{Case}

\newtheorem*{conclusion*}{Conclusion}
\makeatletter
\def\thmheadbrackets#1#2#3{%
  \thmname{#1}\thmnumber{\@ifnotempty{#1}{ }\@upn{#2}}%
  \thmnote{ {\the\thm@notefont[#3]}}}
\makeatother

\newtheoremstyle{brackets}
  {}
  {}
  {\itshape}
  {}
  {\bfseries}
  {.}
  { }
  {\thmheadbrackets{#1}{#2}{#3}}

\theoremstyle{brackets}

\newtheorem*{theorembrackets*}{Theorem}

\makeatletter
\def\thmheadnoparens#1#2#3{%
  \thmname{#1}\thmnumber{\@ifnotempty{#1}{ }\@upn{#2}}%
  \thmnote{ {\the\thm@notefont#3}}}
\makeatother

\newtheoremstyle{noparens}
  {}
  {}
  {\itshape}
  {}
  {\bfseries}
  {.}
  { }
  {\thmheadnoparens{#1}{#2}{#3}}

\theoremstyle{noparens}

\newtheorem*{theoremnoparens*}{Theorem}

\usepackage{xparse}
\DeclareDocumentCommand{\newmathcommand}{mO{0}m}{%
  \expandafter\let\csname old\string#1\endcsname=#1
  \expandafter\newcommand\csname new\string#1\endcsname[#2]{#3}
  \DeclareRobustCommand#1{%
    \ifmmode
      \expandafter\let\expandafter\next\csname new\string#1\endcsname
    \else
      \expandafter\let\expandafter\next\csname old\string#1\endcsname
    \fi
    \next
  }%
}

\newcommand{\C}{\mathbb{C}}

\newmathcommand{\H}{{\mathbb{H}}}

\newcommand{\K}{\mathbb{K}}
\newmathcommand{\L}{\mathbb{L}}

\newmathcommand{\O}{\mathbb{O}}
\newmathcommand{\P}{\mathbb{P}}

\newcommand{\R}{\mathbb{R}}
\newmathcommand{\S}{\mathbb{S}}

\newcommand{\Z}{\mathbb{Z}}

\newmathcommand{\deg}{\operatorname{deg}}

\DeclareMathOperator{\End}{End}

\newmathcommand{\Big}{{\mathrm{Big}}}

\DeclareMathOperator{\im}{im}

\DeclareMathOperator{\JB}{JB}

\newmathcommand{\dv}{\operatorname{div}}
\newmathcommand{\inw}{\operatorname{int}}

\newmathcommand{\epi}{\operatorname{epi}}
\newmathcommand{\diam}{\operatorname{diam}}
\newmathcommand{\diag}{\operatorname{diag}}
\newmathcommand{\regn}{\operatorname{regn}}
\newmathcommand{\relinw}{\operatorname{relint}}

\newmathcommand{\char}{\operatorname{char}}

\DeclareMathOperator{\Id}{Id}

\newmathcommand{\unr}{\mathrm{unr}}

\newmathcommand{\Un}{\operatorname{Un}}

\newmathcommand{\sl}{\mathfrak{s}\mathfrak{l}}

\newmathcommand{\un}{\mathfrak{u}}
\newmathcommand{\su}{\mathfrak{s}\mathfrak{u}}

\makeatletter
\def\@maketitle{%
  \newpage
  \null
  \vskip 2em%
  \begin{center}%
  \let \footnote \thanks
    {\Large\bfseries \@title \par}%
    \vskip 1.5em%
    {\normalsize
      \lineskip .5em%
      \begin{tabular}[t]{c}%
        \@author
      \end{tabular}\par}%
    \vskip 1em%
    {\normalsize \@date}%
  \end{center}%
  \par
  \vskip 1.5em}
  
\makeatother

\title{\sc \huge An order-theoretic characterization of C*-algebras}

\author{Samuel Tiersma%
\thanks{Email: \texttt{s.j.tiersma@math.leidenuniv.nl}}}
\affil{Mathematical Institute, Leiden University, 2300 RA Leiden,
The Netherlands}

\begin{document}

\maketitle
\vspace{-5mm}
\date{}
\vspace{-5mm}

\begin{abstract}
\noindent 
We give an order-theoretic characterization of the essential image of the forgetful functor from the category of real/complex unital C*-algebras to the category of real/complex unital operator systems. It is based on the characterization of JB-algebras among the order unit spaces in terms of the existence of gauge-reversing bijections obtained by M. Roelands and the author in \cite{OCJB}. To this end, we show that a unital operator system is completely order isomorphic to a C*-algebra if and only if each of its matrix spaces admits a compatible JB-algebra structure. As an application, we prove that for $n\ge 4$ the range of a unital $n$-positive projection on a unital real C*-algebra is unitally $n$-order isomorphic to a unital real C*-algebra, which is the analogue of a result proven for complex C*-algebras by Choi–Effros.
\end{abstract}



\section{Introduction}\label{Intro}
C*-algebras form a natural framework for the mathematical formulation of quantum mechanics. For this reason, there has been a substantial interest in characterizations of C*-algebras in terms of the underlying order or Banach space structure \cite{Sher56}, \cite{Sak56}. 
The commutative case of this problem is classical, see e.g.\ \cite{Kad51} and the references therein. Impetus to the study of the noncommutative case was given by Connes' characterization of the ordered vector spaces underlying $\sigma$-finite von Neumann algebras \cite{Connes74}. Alfsen--Shultz gave a convex-geometric characterization of the state spaces of unital C*-algebras \cite{ASchar78}. Characterizations of C*-algebras and related objects in terms of the existence of \emph{holomorphic} (or analytic) symmetries have been given in various categories \cite{UpmHoloC, NlRs02, ElMCRP, ChuGeom}. The main result of this article is the following \emph{order-theoretic} characterization of unital C*-algebras among the operator systems (pertinent definitions are given later in this introduction).

\begin{theorem}\label{MAINTHM} Let $E$ be a real (resp.\ complex) complete unital operator system. Then $E$ is unitally completely order isomorphic to a unital real (resp.\ complex) C*-algebra if and only if for arbitrarily large integers $n\ge 1$ there exists an order-anti-isomorphism $\Phi_n\colon M_n(E)^{\circ}_+ \to M_n(E)^{\circ}_+$ which is homogeneous of degree $-1$ of the interior of the positive cone of $M_n(E)$.
\end{theorem}

We remark that no smoothness assumptions are made on the maps $\Phi_n$, which are known as \emph{gauge-reversing bijections}. The key ingredient in the proof is a characterization of JB-algebras among the order unit spaces obtain by M. Roelands and the author in \cite{OCJB}. We shall first discuss this result for JB-algebras, and then make the connection with operator systems.


\ind The prototypical example of a JB-algebra is the self-adjoint part $A_{sa}$ of a real or complex C*-algebra $A$, endowed with the nonassociative \emph{Jordan product} $x \circ y := \frac{1}{2}(xy+yx)$. If $J$ is a JB-algebra with identity $1_J$, then $J$ is partially ordered by its cone of squares $J_+ = \{x^2\colon x \in J\}$ and $(J, J_+, 1_J)$ is an order unit space. The interior $J_+^{\circ}$ of $J_+$ consists of the positive invertible elements. We record the following observation which will be of cardinal importance. The inversion map $i\colon J_+^{\circ} \to J_+^{\circ}$, $i(x) = x^{-1}$ is an order-anti-isomorphism which is homogeneous of degree $-1$ (see e.g.\ \cite[Lemma 1.31]{AS03}).\\
\ind Now let $(V, V_+, v)$ be an order unit space and denote $V^{\circ}_+$ the interior of its positive cone. A mapping $\Phi \colon V^{\circ}_+ \to V^{\circ}_+$ is called a \emph{gauge-reversing bijection} if it is an order-anti-isomorphism and homogeneous of degree $-1$, that is, for all $x,y \in V_+^\circ$ and $\lambda > 0$ it holds that
\begin{equation}
x \le y \iff \Phi(x) \ge \Phi(y)\quad \text{ and }\quad \Phi(\lambda x)=\lambda^{-1}\Phi(x).
\end{equation}

Let us call a \emph{compatible JB-algebra structure on} $(V, V_+, v)$ to be a structure of JB-algebra on $V$ for which $v$ is the identity element and $V_+ = \{x^2: x\in V\}$ the cone of squares; it is unique if it exists \cite[Thm. 2.80]{AS03}. In view of the above observation about the inversion map in a JB-algebra, a necessary condition for the existence of a compatible JB-algebra structure on $(V, V_+, v)$ is the existence of a gauge-reversing bijection on $V_+^\circ$. Remarkably, this conditions is also sufficient.

\begin{theorem*}[\cite{OCJB}] A complete order unit space $(V, V_+, v)$ has a compatible JB-algebra structure if and only if there exists a gauge-reversing bijection $\Phi\colon V^{\circ}_+ \to V^{\circ}_+$.
\end{theorem*}

We refer to the introduction of \cite{OCJB} for a discussion of this and related characterizations of Jordan structures, and will now turn to C*-algebras.\\
\ind Let $E = (E, *, \{M_n(E)_+\}_n, e)$ be a unital operator system, i.e.\ $E$ is a matrix-ordered *-vector space with Archimedean matrix order unit $e$, defined over the base field $\K \in \{\R, \C\}$. Each unital C*-algebra $A$ over $\K$ defines a unital operator system as follows from the fact that $M_n(A) = A \otimes_\K M_n(\K)$ is again a C*-algebra. Hence $M_n(A)_{sa}$ is a JB-algebra whose inversion map defines a gauge-reversing bijection $i_n\colon M_n(A)_+^{\circ} \to M_n(A)_+^{\circ}$. This shows necessity of the condition in \Cref{MAINTHM} for $E$ to be isomorphic, in the category of unital operator systems, to a unital C*-algebra $A$. We now give an outline of the proof of sufficiency.\\
\ind First, the order-theoretic characterization of JB-algebras is invoked to obtain a compatible JB-algebra structure on $(M_n(E)_{sa}, M_n(E)_+, e \otimes I_n)$, whose product we write as $(x, y) \mapsto x \circ y$. It turns out (cf. \Cref{circdot}) that the product $\circ$ is \emph{automatically compatible} with the operator system structure, in the sense that for all $a\in M_n(\K)_{sa}$ and $x\in M_n(E)_{sa}$ it holds that
$$a \circ x = \frac{1}{2}(a \cdot x + x \cdot a),$$
where the right hand side is given by the $M_n(\K)$-bimodule structure on $M_n(E)$. This equality is deduced from an order-theoretic characterization of compressions in a unital JB-algebra in terms of bicomplementary projections due to Alfsen--Shultz.\\
\ind We are then prepared to apply coordinatization theorems for JB-algebras containing an isomorphic copy of $M_n(\R)_{sa}$ or $M_n(\C)_{sa}$, due to Hanche--Olsen resp.\ Jacobson. This yields a JB-isomorphism $M_n(E)_{sa} \cong M_n(D)_{sa}$. The `coordinate' algebra $D$ is an associative unital *-algebra with a proper involution, whose self-adjoint part $D_{sa}$ is a JB-algebra, and is shown to be a C*-algebra in the norm $\lVert x\rVert_D = \lVert x^*x\rVert_{D_{sa}}^{1/2}$. In this way we shall obtain the following result, which is an alternative formulation of our main theorem.

\begin{theorem}\label{MAINTHM2} Let $(E, *, \{M_n(E)_+\}_n, e)$ be a (real or complex) unital operator system. Let $m \in \Z$ be an integer such that $m\ge 2$ if $E$ is complex and $m\ge 4$ if $E$ is real. Then:\\
(1) $E$ is unitally $m$-order isomorphic to a C*-algebra if and only if there exists a compatible \mbox{JB-algebra} structure on the order unit space $(M_m(E), M_m(E)_+, e \otimes I_m)$.\\
(2) $E$ is unitally completely order isomorphic to a C*-algebra if and only if  there exists a compatible JB-algebra structure on the order unit space $(M_n(E), M_n(E)_+, e \otimes I_n)$ for arbitrarily large $n\ge 1$.
\end{theorem}

This theorem is inspired by and should be compared with the holomorphic characterization of complex C*-algebras obtained by Neal--Russo in \cite{NlRs02}. They showed that for each integer $m\ge 2$ a complex operator space $E$ is $m$-isometrically isomorphic to a C*-algebra if and only if the open unit ball of $M_m(E)$ is a bounded symmetric domain of tube type. The latter condition is equivalent to the Banach space $M_m(E)$ being isometric to a JB*-algebra, i.e.\ to the complexification of a JB-algebra.\\
\ind We would like to point out that any complete order isomorphism between operator systems is a complete isometry, so that Neal and Russo's result implies part (2) of \Cref{MAINTHM2} in the complex case. However, operator systems can be $n$-order isomorphic without being $n$-isometrically isomorphic, so a new approach seems necessary for part (1).
Moreover, their proof, which takes place in the categories of JB*-triples and ternary rings of operators, does not seem to apply to real C*-algebras. By contrast, the JB-algebra framework allows for a uniform treatment of complex and real C*-algebras alike. This opens the possibility of proving theorems at once for real and complex C*-algebras as a corollary of the analogous result for JB-algebras. We illustrate this method on the Choi--Effros theorem concerning the range of a unital completely positive projection on a C*-algebra. This approach also yields a version of that theorem for a unital $n$-positive projection on a unital real C*-algebra, where $n\ge 4$.
\\ 

\section{Preliminaries}
This section contains preliminary results on order unit spaces, operator systems, real C*-algebras and JB-algebras.

\subsection{Order unit spaces and gauge-reversing bijections}

Let $(V, \le)$ be a partially ordered vector space with positive cone $V_+$.
A positive element $v\in V_+$ is called an \emph{order unit} if for every $x\in V$ there exists $\lambda \ge 0$ such that $x \le \lambda v$. An order unit $v$ is called \emph{Archimedean} if $x \le \lambda v$ for all $\lambda \ge 0$ implies $x\le 0$. An \emph{order unit space} is a triple $(V, V_+, v)$ consisting of a partially ordered vector space and an Archimedean order unit.
One defines a norm on $V$, called the \emph{order unit norm}, by 
\begin{equation}
\lVert x \rVert_v := \inf \{\lambda \ge 0: -\lambda v \le x \le \lambda v\}.
\end{equation}
The order unit space $(V, V_+, v)$ is said to be \emph{complete} if $(V, \lVert\cdot\rVert_v)$ is a Banach space. Each positive linear map $T\colon V \to V$ is automatically bounded with respect to the order unit norm, and the $\lVert\cdot\rVert_v$-operator norm is given by $\lVert T\rVert = \lVert T(v)\rVert_v$.  \\ \ind The $\lVert\cdot\rVert_v$-interior of $V_+$ is denoted $V_+^{\circ}$ and referred to as the \emph{open cone} of $V_+$. An element $x\in V$ belongs to $V_+^{\circ}$ if and only if there exists $\delta > 0$ such that $x \ge \delta v$. It follows that $V^\circ_+$ consists precisely of the order units of $(V, V_+)$. Therefore there exists a well-defined \emph{gauge-function} $M\colon V_+^{\circ} \times V_+^{\circ} \to \R_{>0}$ given by
\begin{equation}\label{MDef}
    M(x/y) := \inf \{\lambda > 0: x \le \lambda y\}.
\end{equation}
\vspace{-1em}
\begin{definition}\label{grDef}
A map $\Phi\colon V_+^{\circ} \to V_+^{\circ}$ is called \emph{gauge-reversing} if it reverses gauges, i.e.
\begin{equation}
    M(\Phi(x)/\Phi(y))=M(y/x) \text{ for all }x,y \in V_+^{\circ}.
\end{equation}
\end{definition}
To give an equivalent definition, let us say that $\Phi\colon V_+^{\circ} \to V_+^{\circ}$ is \emph{homogeneous of degree} $-1$ if $\Phi(\lambda x)=\lambda^{-1}\Phi(x)$ for all $\lambda > 0$ and $x\in V_+^{\circ}$. We call $\Phi$ an \emph{order-anti-isomorphism} if it is bijective and $\Phi(x) \le \Phi(y) \iff y \le x$ holds for all $x,y\in V_+^{\circ}$.

\begin{lemma} Let $(V, V_+, v)$ be an order unit space. A mapping $\Phi\colon V_+^{\circ} \to V_+^{\circ}$ is a gauge-reversing bijection if and only if it is an order-anti-isomorphism and homogeneous of degree $-1$.
\end{lemma}
\begin{proof}
See for example \cite[Lemma 2.5]{LRW25}
\end{proof}

It is known that gauge-reversing bijections are automatically infinitely differentiable and, in fact, real-analytic \cite[Thm. 5.20]{OCJB}. Each gauge-reversing bijection $\Phi$ is an isometry for 
\emph{Thompson's metric}, defined by
\begin{equation}\label{dT}
d_T(x, y) := \log \max \{M(x/y), M(y/x)\},\quad x,y\in V_+^\circ.
\end{equation}
A \emph{Thompson symmetry} of $V_+^\circ$ at a point $p \in V_+^\circ$ is an involutive isometry of $(V_+^\circ, d_T)$ having $p$ as an isolated fixed point. It was shown in \cite[Theorem 3.5]{LRW25} that each Thompson symmetry is a gauge-reversing bijection. Conversely, if $V_+^\circ$ admits any gauge-reversing bijection, then at each point of $V_+^\circ$ there exists a unique Thompson symmetry  \cite[Lemmas 5.7, 5.9 and 5.10]{OCJB}.\\
\ind The relation between gauge-reversing bijections and JB-algebras is discussed in \Cref{JBprel}.  

\subsection{Operator systems}
We briefly introduce operator systems. We refer to the book \cite{Paucboa} for results on complex operator systems and to \cite[Sect. 3]{BleRus} for a systematic discussion which of those results remain valid for real operator systems.\\

Given a vector space $F$ over a field $\K$ and integers $m,n \ge 1$, we write $M_{mn}(F)$ for the $F$-vector space of $m \times n$-matrices with entries in $F$. In case $m=n$ we simply write $M_n(F)$. The unit element of the $\K$-algebra $M_n(\K)$, i.e.\ the $n \times n$-identity matrix, is denoted $I_n$. Note that $M_{mn}(F)$ has a natural $M_m(\K)$-$M_n(\K)$-bimodule structure. More generally, for integers $k,l,m,n\ge 1$ matrix multiplication using the $\K$-vector space structure on $F$ defines a map
\begin{align*}
M_{k,m}(\K) \times M_{m,n}(F) \times M_{n,l}(\K) &\to M_{k,l}(F),\\
(a, x, b) &\mapsto a \cdot x \cdot b.
\end{align*}
Although it is tempting to denote the bimodule action of $M_n(\K)$ on $M_n(E)$ by mere juxtaposition, we shall refrain from doing so at this stage, the reason being that we shall consider on $M_n(E)$ also certain Jordan and associative products, whose compatibility with the bimodule action is yet to be established.
Without further comment, we identify $M_{mn}(F)$ with the tensor product $F \otimes_{\K} M_{mn}(\K)$.\\
\ind An \emph{involution} $*$ on a real (resp. complex) vector space $F$ is a linear (resp.\ conjugate-linear) map $*\colon F \to F, x \mapsto x^*$ which is equal to its own inverse, i.e.\ $(x^*)^*=x$ for every $x\in V$. An involution on $F$ induces an involution on $M_n(F)$  given by $(x^*)_{ij} = (x_{ji})^*$ for each $x = (x_{ij}) \in M_n(F)$. We write $M_n(F)_{sa} = \{x\in M_n(F)\colon x = x^*\}$ for the real subspace of self-adjoint elements in $M_n(F)$. 

\begin{definition} Let $\K \in \{\R, \C\}$. An \emph{operator system over} $\K$ is a quadruple $(E, *, \{M_n(E)_+\}_n, e)$ consisting of a $\K$-vector space $E$, an involution $*\colon E \to E$, a sequence of cones $M_n(E)_+ \subset M_n(E)_{sa}$ and an element $e \in E$ such that for every $m \times n$-matrix $a \in M_{mn}(\K)$ one has
\begin{equation}\label{matrOrd}
a^{*} \cdot M_m(E)_+ \cdot a \subset M_n(E)_+
\end{equation}
and $e \otimes I_n \in M_n(E)$ is an Archimedean order unit of $(M_n(E)_{sa}, M_n(E)_+)$ for all $n,m \in \Z_{\ge 1}$. 
\end{definition}

Throughout this article, $M_n(\K)$ is considered to be embedded in $M_n(E)$ via the map $a \mapsto e \otimes a$. In particular, we write $I_n$ instead of $e \otimes I_n$ for the order unit of $M_n(E)_{sa}$. Owing to the presence of this order unit, the real linear span of $M_n(E)_+$ equals $M_n(E)_{sa}$, so in the complex case the involution $*$ is determined by the positive cone $M_n(E)_+$.\\
\ind Let $(F, *, \{M_n(F)_+\}_n, f)$ be a second operator system over the same field $\K$ and let $\phi\colon E \to F$ be a linear map. We say that $\phi$ is \emph{unital} if $\phi(e) = f$ and \emph{self-adjoint} if $\phi(x^*)=\phi(x)^*$ for every $x\in E$. Let $n$ be a positive integer. Entrywise application of $\phi$ defines a map
\begin{align*}
\phi_n\colon M_n(E) &\to M_n(F),\\
(x_{ij})_{ij} &\mapsto (\phi(x_{ij}))_{ij}.\notag 
\end{align*}
A self-adjoint linear map $\phi\colon E \to F$ is called $n$\emph{-positive} (resp.\ an $n$\emph{-order embedding}, resp.\ an $n$\emph{-order isomorphism}) if the map $\phi_n$ is positive (resp.\ an order embedding, resp.\ an order isomorphism). It follows readily from \eqref{matrOrd} that when $m\le n$ and $\phi_n$ has one of these properties, so does $\phi_m$. We call $\phi$ \emph{completely positive} (resp.\ a complete order embedding, etc.) if $\phi_n$ is positive (resp.\ an order embedding, etc.) for every $n \ge 1$. A \emph{state} on $E$ is a unital positive (self-adjoint) linear map $\phi\colon E \to \K$. Each state $\phi$ is completely positive (for real operator systems, this is proved in \cite[Lemma 5.14]{BleTep}).\\
\ind If $E$ is a self-adjoint unital subspace of bounded operators on a Hilbert space $H$ over $\K$, then $(E, *, \{M_n(E)_+\}_n, \Id_H)$ is an operator system where * is given by the Hilbert space adjoint, the cone $M_n(E)_+$ consists of the positive-definite operators on the $n$-fold direct sum $H^{\oplus n}$, and $\Id_H$ is the identity on $H$. Conversely, for each operator system $(E, *, \{M_n(E)_+\}_n, I)$ there exists a Hilbert space $H$ over $\K$ and a unital (self-adjoint) complete order embedding $\pi\colon E \to B(H)$ (the complex case is due to Choi--Effros \cite[Thm. 4.4]{ChEff77}, the real case is the variant of a result by Ozawa \cite[Section 17]{Ozawa} given by Blecher--Russell in \cite[Thm. 2.7]{BleRus}).\\
\ind It follows that there exists a sequence of norms $\lVert\cdot\rVert_n$ on $M_n(E)$, extending the order unit norm $\lVert\cdot\rVert_{I_n}$ on $M_n(E)_{sa}$, such that each unital complete order embedding of $E$ into $B(H)$ is a complete isometry. These norms satisfy for all $a\in M_{m,n}(\K)$, $b \in M_{n,m}(\K)$ and $x \in M_{n}(E)$ 
\begin{equation}\label{mtrNorm}
    \lVert a \cdot x \cdot b\rVert_{m} \le \lVert a \rVert \lVert x\rVert_{n} \lVert b \rVert.
\end{equation}

\subsection{Real C*-algebras}
We assume the reader to be familiar with the basic theory of complex C*-algebras, but will give a brief discussion of real C*-algebras based on \cite{Schr93} and \cite{Ros16}. For a more in-depth treatment of real C*-algebras, the reader may consult \cite{Gdl82}.

\begin{definition}\label{realCDef}
A \emph{real $C^*$-algebra} is a real associative $*$-algebra $A$ which is complete for a norm $\lVert\cdot\rVert$ satisfying for all $a,b\in A$:
\begin{enumerate}
\item $\lVert ab\rVert \le \lVert a\rVert \lVert b\rVert$,
\item $\lVert a^*a \rVert = \lVert a \rVert^2$, and,
\item $\lVert a^*a \rVert \le \lVert a^*a + b^* b \rVert$ .
\end{enumerate}
\end{definition}
Conditions (i)-(iii) may be replaced by the requirement that $A$ be isometrically isomorphic to a norm closed *-algebra of bounded operators on a real Hilbert space. If $A$ is a real C*-algebra, then its complexification $A_{\C} := A \otimes_\R \C$ becomes a complex C*-algebra. On the other hand, if a complex C*-algebra $B$ admits an involutive complex-linear $*$-antiautomorphism $\theta$, then $B$ is the complexification $A_\C$ of the real C*-algebra $A := \{x\in B\colon \theta(x) = x^*\}$. Moreover, each complex C*-algebra $B$ becomes a real C*-algebra by restriction of scalars.\\
\ind The category of commutative real C*-algebras is equivalent to the category of pairs $(X, \tau)$ where $X$ is a locally compact Hausdorff space and $\tau$ is an involutive autohomeomorphism of $X$. This equivalence is implemented by the functor sending the pair $(X, \tau)$ to the real C*-algebra $$C_0(X, \tau) := \{f \in C_0(X): f(\tau(x)) = \overline{f(x)} \text{ for all }x \in X\}.$$

The following definition should be clear.

\begin{definition} A (real or complex) normed *-algebra $A$ is called a \emph{pre-C*-algebra} if it satisfies all conditions to be a C*-algebra except that its norm need not be complete.
\end{definition}

It is straightforward to verify that the norm completion of a real (resp.\ complex) pre-C*-algebra is a real (resp.\ complex) C*-algebra.


\subsection{JB-algebras} \label{JBprel}
We give the definition of and preliminary results on JB-algebras.\\

A real \emph{Jordan algebra} is a real vector space $J$ endowed with a bilinear commutative (not necessarily associative) product $(a, b) \mapsto a \circ b$, which satisfies for all $a,b\in J$ the \emph{Jordan identity}
\begin{equation}\label{Jd}
a \circ (b \circ a^2) = (a \circ b) \circ a^2.
\end{equation}
We call $J$ \emph{unital} if it has an \emph{identity element} $1$, i.e.\ $1 \circ a = a$ for all $a \in J$. Each element $a\in J$ defines a linear map $T_a \in \End_\R(J)$ given by
\begin{equation}\label{TaDef}
T_a(x) = a \circ x  
\end{equation}
and a quadratic representation map
\begin{equation}\label{UaDef}
    U_a := 2T_a^2 - T_{a^2} \in \End_\R(J).
\end{equation}
Furthermore, for $a,b,c\in J$ we define the \emph{triple product}
\begin{equation}\label{tripDef}
    \{a, b, c\} := a \circ (b \circ c) + (a \circ b) \circ c - b \circ (a \circ c).
\end{equation}
Since $\{a, b, a\} = U_a(b)$, the triple product is the linearization of the quadratic product given by $(a, b) \mapsto U_a(b)$. Elements $a,b\in J$ are said to \emph{operator commute} if $T_aT_b = T_bT_a$. A Jordan algebra is called \emph{associative} if its product $\circ$ is associative, equivalently, if any two of its elements operator commute.

\begin{definition}
A \emph{JB}-algebra is a real Jordan algebra $A$ which is complete for a norm $\lVert \cdot \rVert$ satisfying for all $a,b\in A$:
\begin{enumerate}
    \item $\lVert a\circ b\rVert \le \lVert a\rVert \lVert b\rVert$;
    \item $\lVert a^2\rVert = \lVert a\rVert^2$;
    \item $\lVert a^2\rVert \le \lVert a^2+b^2\rVert$.
\end{enumerate}
\end{definition}
Each norm closed Jordan subalgebra $B$ of a JB-algebra $A$ is again a JB-algebra for the restriction of the norm. We call $B$ a \emph{JB-subalgebra} of $A$.

\begin{example}[JC-algebras]\label{JCExp} Let $A$ be a real or complex C*-algebra, and denote its real linear subspace of self-adjoint elements by $A_{sa} := \{x\in A: x^* = x\}$. Let $J$ be a norm closed subspace of $A_{sa}$ which is closed under taking squares. Then $J$ is a JB-algebra in the \emph{Jordan product} 
\begin{equation}
    a \circ b = \tfrac{1}{2}(ab+ba),
\end{equation}
with respect to the restriction of the norm of $A$  to $J$. Indeed, the norm axioms follow directly from Definition \ref{realCDef}(i-iii) by taking $a$ and $b$ there to be self-adjoint. The Jordan identity is satisfied since for any $a,b,c\in J$ we have
\begin{equation}\label{TaTb}
[T_a, T_b](c) = \tfrac{1}{4}[[a,b],c].
\end{equation}
In fact, $a$ and $b$ operator commute in $J$ if and only if they commute in $A$, i.e.\ $[a,b]:=ab-ba=0$ \cite[Prop. 1.49]{AS03}.
The quadratic and triple products in $J$ are given by
\begin{align}\label{qtrsp}
    U_a(b) &= aba,\\
    \{a,b,c\}& = \frac{1}{2}(abc+cba). 
\end{align}
JB-algebras $J$ of this kind are known as \emph{JC-algebras}.
\end{example}

Let $J$ be a JB-algebra. For a subset $S\subset J$, the set of those elements in $J$ which operator commute with all elements of $S$ is denoted
\begin{equation}
S' := \{a\in J: [T_a, T_b]=0 \text{ for each }b \in S\}
\end{equation}
and called the \emph{operator commutant} of $S$.

\begin{lemma}\label{OpCmm} Let $S$ be a subset of a JB-algebra $J$. Then the operator commutant $S'$ is a JB-subalgebra of $J$.
\end{lemma}
\begin{proof}
Bilinearity and continuity of the product readily imply that $S'$ is a closed linear subspace of $J$. It follows from \cite[Theorem 3.13(a),(d)]{vdWopCm} that $S'$ is closed under taking squares. Therefore, $S'$ is a JB-subalgebra of $J$.
\end{proof}

A JB-algebra $J$ is partially ordered by its \emph{cone of squares}
\begin{equation}\label{Jpl}
    J_+ := \{x^2: x \in J\}.
\end{equation}
For every $x\in J$ the quadratic representation map $U_x\colon J \to J$ is positive with operator norm $\lVert U_x\rVert = \lVert x\rVert^2$ \cite[Thm. 1.25]{AS03}. If $J$ is a nonunital JB-algebra, then $\tilde{J} := J \oplus \R 1$ endowed with the product $(a + \lambda 1) \circ (b + \mu 1) = (a \circ b + \lambda b + \mu a) + \lambda\mu 1$ is a unital JB-algebra \cite[Thm. 3.3.9]{HOSt84} which (isometrically) contains $J$ as a JB-subalgebra.\\
\ind Now assume that $J$ has an identity element $1$. Then $(J, J_+, 1)$ is an order unit space, and the order unit norm coincides with the JB-algebra norm. An element $a\in J$ is called \emph{invertible} if there exists an element $b\in J$ which operator commutes with $a$ and satisfies $a\circ b = 1$. The element $b$, if it exists, is unique and will be denoted $b := a^{-1}$. A positive element $a\in J_+$ is invertible if and only if it belongs to the open cone of $J$, i.e.\ $a \in J_+^{\circ}$, in which case $a^{-1} \in J_+^{\circ}$ as well.

\begin{lemma}\label{invgr}
The inversion map $i\colon J_+^{\circ} \to J_+^{\circ}$, $i(x) = x^{-1}$ is a gauge-reversing bijection. 
\end{lemma}
\begin{proof}
Clearly $i$ is homogeneous of degree $-1$. The proof that $i$ is an order-anti-isomorphism is given in \cite[Lemma 1.31]{AS03}.
\end{proof}

The \emph{spectrum} of an element $a$ in the unital JB-algebra $J$ is the set
\begin{equation}
    \sigma_J(a) := \{\lambda \in \R: \lambda 1 - a \text{ is not invertible in }J\}.
\end{equation}
The spectrum $\sigma_J(a)$ is a non-empty compact subset of $\R$ and $\lVert a\rVert = \max\{|\lambda|: \lambda \in \sigma(a)\}$. The Jordan algebra structure of a unital JB-algebra is uniquely determined by the underlying order unit space structure, as the following theorem shows.
\begin{theorem}\label{Kap}
Let $\Phi\colon H \to J$ be a unital order isomorphism between unital JB-algebras. Then $\Phi$ is an (isometric) isomorphism of JB-algebras.
\end{theorem}
\begin{proof}
See \cite[Thm. 2.80]{AS03}.
\end{proof}

Let $(V, V_+, v)$ be a complete order unit space. A \emph{compatible JB-algebra structure} on $(V, V_+, v)$ is a structure of JB-algebra on $V$ for which $v$ is an identity element and the cone of squares equals $V_+$. It follows from \Cref{Kap} that there can be at most one compatible JB-algebra structure on $(V, V_+, v)$. In view of \Cref{invgr}, a necessary condition for existence is that $V_+^{\circ}$ admit a gauge-reversing bijection.  M. Roelands and the author have shown that this condition is also sufficient in \cite{OCJB}. There, the following order-theoretic characterization of the class of order unit spaces admitting a compatible JB-algebra structure was proven.

\begin{theorem}[\cite{OCJB}]\label{MRJB}
A complete order unit space admits a compatible JB-algebra structure if and only if the interior of its positive cone admits a gauge-reversing bijection.
\end{theorem}

\subsubsection{Positive projections}
In this section, we state for later use two results concerning positive projections on a JB-algebra.\\

Let $F$ be an ordered normed space. Two positive idempotent linear maps $P$ and $Q$ on $F$ are called \emph{complementary} if
\begin{equation}
    \ker_+ P = \im_+ Q\text{ and } \ker_+ Q = \im_+ P.
\end{equation}
Here we set $\ker_+ P := F_+ \cap \ker P$ and $\im_+ P := F_+ \cap \im P = P(F_+)$. The norm dual $F^*$ of $F$ has a natural ordering given by the cone $F^*_+ := \{\phi \in F: \phi(F_+) \subset \R_+\}$. Two contractive complementary projections $P, Q\colon F \to F$ whose duals $P^*, Q^*\colon F^* \to F^*$ are also complementary are called a pair of \emph{bicomplementary projections} on $F$.\\
\ind If $p$ and $q := 1-p$ are complementary projections in a unital JB-algebra $J$, then their quadratic representation maps $U_p, U_q\colon J \to J$ are bicomplementary projections \cite[Thm. 7.12]{AS03}. Moreover, any pair of bicomplementary projections on a unital JB-algebra arises in this way.

\begin{theorem}[Alfsen--Shultz]\label{BicpThm}
Let $J$ be a unital JB-algebra. Let $P, Q\colon J \to J$ be bicomplementary projections. Then $p = P(1)$ and $q = Q(1)$ are complementary projections with $P = U_p$ and $Q = U_q$.
\end{theorem}
\begin{proof}
Cf. \cite[Thm. 2.83]{AS03} and the remark following its proof.
\end{proof}

The range of a positive unital projection is itself a JB-algebra under the Choi--Effrøs product.

\begin{theorem}[Effros-Størmer]\label{EffSt}
Let $J$ be a unital JB-algebra and suppose that $P\colon J \to J$ is a unital positive idempotent linear map. Then $P(J)$ is a unital JB-algebra under the product $(r, s) \mapsto P(r \circ s)$.
\end{theorem}
\begin{proof}
See \cite[Thm. 1.4]{EffSto79}
\end{proof}

\subsubsection{Coordinatization theorems}
As follows from \Cref{JCExp}, if $D$ is a real or complex C*-algebra and $n\ge 1$ is an integer, then $M_n(D)_{sa}$ is a JB-algebra. In this section, we discuss two \emph{coordinatization theorems} for JB-algebras, which give a sufficient condition in order that a JB-algebra is isomorphic to $M_n(D)$ for some real or complex *-algebra $D$. To show that the coordinate *-algebra $D$ is in fact a real or complex C*-algebra, we require a preliminary result.\\
\ind Alfsen--Schultz proved in \cite[Thm. 1.96]{AS01} that if $A$ is a complex unital *-algebra $A$ whose self-adjoint part $J := A_{sa}$ is a JB-algebra for some norm $\lVert\cdot\rVert_J$, then $A$ is a C*-algebra for a unique norm extending $\lVert\cdot\rVert_J$. For a real *-algebra $A$, the analogous result is not true, cf.\ \Cref{PthEg}.
However, if we make the additional assumption that the involution on $A$ is \emph{proper}, i.e.\ $x^*x=0$ implies $x=0$ for $x\in A$,  we can show that $A$ becomes a real pre-C*-algebra for a unique, not necessarily complete, norm extending $\lVert\cdot\rVert_J$.


\begin{theorem}\label{norm}
Let $A$ be an associative real (resp.\ complex) $*$-algebra such that $J := A_{sa}$ with the induced Jordan product is a JB-algebra for some norm $\lVert\cdot\rVert_J$. In the real case, assume that the involution is proper. Then setting $\lVert x\rVert := \lVert x^*x \rVert_J^{1/2}$ defines a norm on $A$ which extends $\lVert \cdot \rVert_J$ and makes $A$ into a real pre-C*-algebra (resp.\ a complex C*-algebra).
\end{theorem}

The various steps in the proof are essentially adaptions from the arguments given in \cite[Lemma 1.92, Thm. 1.96, Prop. 1.100 and 1.103]{AS01} and \cite[Lemma 4.3]{AS03}, barring the real case of step 3. In this step we invoke the automatic continuity of Jordan derivations from a JB-algebra into itself. Recall that a \emph{Jordan derivation} of a JB-algebra $J$ into itself is a map $D\colon J \to J$ such that $D(a \circ b) = D(a) \circ b + a \circ D(b)$ for all $a,b\in J$. 

\begin{lemma}\label{Dct} Let $J$ be a JB-algebra. Each Jordan derivation $D\colon J \to J$ is continuous.
\end{lemma}
\begin{proof}
As has been observed in \cite[Sec. 1.1]{UpmDer}, the proof that associative derivations on C*-algebras are norm continuous given in \cite[Lemma 4.1.3]{SakCW} may be modified to prove the lemma.
\end{proof}

\begin{proof}[Proof~of~\Cref{norm}.]
\textbf{Step 1.} If $A$ is not unital, then we adjoin a unit element $1$ to $A$, to obtain an associative real *-algebra $\tilde{A} = A \oplus \R 1$ with identity element $1$. Moreover, the unitization of the JB-algebra $J$ is a JB-algebra $\tilde{J} = J \oplus \R 1$. Since we have a Jordan isomorphism $\tilde{A}_{sa} \cong A_{sa} \oplus \R1 = J$, we are reduced to the the unital case.\\ 
\textbf{Step 2.} For $x\in J$ we have $\sigma_J(x) = \sigma_{A}(x)$.\\
It suffices to prove that $x$ is invertible in $A$ if and only if $x$ has a Jordan inverse in $J$. Suppose that $y \in A$ satisfies $xy=1=yx$. Applying the involution yields $y^*x=1=xy^*$. Hence $y^*$ is also an associative inverse to $x$; by uniqueness $y=y^*$. Since $x$ and $y$ associate in $J$ and $x\circ y = \frac{1}{2}(xy+yx)=1$, we conclude that $y$ is the Jordan inverse of $x$ in $J$.\\
\ind Conversely, assume that $y\in J$ is a Jordan-inverse of $x$. Then $x$ and $y$ associate in $J$, hence commute in $A$. This gives that $xy=yx = x\circ y = 1$, proving that $y$ is invertible in $A$. \\
\textbf{Step 3.} If $x\in A$ satisfies $-x^*x\in J_+$, then $x=0$.\\
According to a general fact about unital associative algebras, if $a,b\in A$, then $1-ab$ is invertible if and only if $1-ba$ is invertible. This fact entails that $\sigma_{A}(-x^*x)\setminus \{0\} = \sigma_{A}(-xx^*)\setminus \{0\}$. The previous claim yields $\sigma_J(-x^*x)\setminus \{0\}=\sigma_J(-xx^*)\setminus \{0\}$. Since the positive cone in the JB-algebra $J$ is described as $J_+ = \{a \in A \colon \sigma_J(a) \subset \R_+\}$, it follows that $-xx^* \in J_+$ and thus $x^*x+xx^* \in -A_+$.\\
\textit{Complex case.} Write $x = a + ib$ with $a,b \in J$. We find $$x^*x + xx^* = (a-ib)(a+ib)+(a+ib)(a-ib)= 2a^2 + 2b^2 \in J_+ \cap -J_+ = \{0\}.$$ Since $J$ is formally real, $a = b =0 $ and $x = 0$ follow.\\
\textit{Real case.} Put $a = \frac{1}{2}(x+x^*)$ and $b=\frac{1}{2}(x-x^*)$, so that $x = a + b$ with $a = a^*$ and $b = -b^*$. Then we have $x^*x + xx^* = (a-b)(a+b) + (a+b)(a-b) = 2a^2 - 2b^2$. Since $a\in J$, we have $a^2 \ge 0$. This shows that $b^2 = a^2 - \frac{1}{2}(x^*x + xx^*) \in J_+$.
Hence there exists a square root $u \in \JB(B^2)$ of $b^2$ in the JB-subalgebra $\JB(b^2)$ of $J$ generated by $b^2$. We contend that $u$ and $b$ commute in $A$.\\
\ind Consider the real linear map $D_A \colon A \to A$ given by $D_A(y) := [b,y] := by - yb$. Then $D_A(J)\subset J$ holds since for every $y \in J$ we have $[b,y]^*=[y^*,b^*]=[y,-b]=[b,y]$. Let $D_J\colon J \to J$ be the restriction of $D_A$ to $J$. Since $D_A$ is an associative derivation on $A$, we have that $D_J$ is a Jordan derivation from $J$ into itself. It follows that $\ker D_J$ is a Jordan subalgebra of $J$. Moreover, $D_J$ is continuous by \Cref{Dct}, so $\ker D_J \subset J$ is closed, hence a JB-subalgebra of $D$. Since $b^2 \in \ker D_J$ and $\JB(b^2)$ is the smallest JB-subalgebra of $J$ containing $b^2$, it follows that $\JB(b^2) \subset \ker D_J$. Since $u \in \JB(b^2)$, we find that $bu-ub=D_J(u)=0$. Therefore $u$ and $b$ commute in $A$, as contended.\\
\ind Now let $y = u + b$. Then we have that $y^*y = (u - b)(u + b)=u^2 - b^2 = 0$. It follows that $y = 0$, hence $b = \frac{1}{2}(y-y^*)=0$ and $x = a \in J$. Since $J$ is formally real, $x\in J$ and $x^2 = x^*x \in -J_+$ imply $x = 0$, proving the claim.\\
\textbf{Step 4.} If $x\in A$ then $x^*x \in J_+$.\\
Let $x^*x = a^+ - a^-$ be the Jordan decomposition of $x^*x \in J$, where $a^+,a^-\in J$ are orthogonal. Let $y = xa^-$. Then $-y^*y = -a^-(x^*x)a^- = -a^-(a^+-a^-)a^-=(a^-)^3 \in A_+$. By the previous claim, we have $y = 0$. Now $a^- \in J$ cubes to $0$, hence $x^*x = a^+ - a^- = a^+ \in J_+$.\\
\textbf{Step 5.} $\lVert\cdot\rVert$ is a norm on $A$ extending $\lVert\cdot\rVert_J$.\\
Let $x,y\in A$ and $\lambda \in \K$. For $x\in J$ we have $\lVert x\rVert := \lVert x^*x\rVert_J^{1/2} = \lVert x^2\rVert_J^{1/2} = \lVert x\rVert_J$ since $\lVert\cdot\rVert_J$ is a JB-norm. Therefore, the norm $\lVert \cdot \rVert$ extends $\lVert\cdot\rVert_J$.  Also,  $(\lambda x)^*=\overline{\lambda}x^*$
yields $\lVert \lambda x\rVert = |\lambda|\lVert x\rVert$. 
To prove the triangle inequality, consider a state $\rho$ on $J$, i.e.\ a positive linear functional with $\rho(1) = 1$. It follows from step 4 that a semi-inner product on $J$ is defined by $(z|w) := \rho(z^*w)$. It follows from the Cauchy--Schwartz inequality that $|\rho(x^*y)|\le \rho(x^*x)^{1/2}\rho(y^*y)^{1/2}$. This yields
\begin{equation*}
\rho((x+y)^*(x+y))\le \rho(x^*x)+\rho(y^*y) + 2\rho(x^*x)^{1/2}\rho(y^*y)^{1/2}=(\rho(x^*x)^{1/2}+\rho(y^*y)^{1/2})^2 \le (\lVert x\rVert + \lVert y\rVert)^2.
\end{equation*}
Taking the supremum over all states $\rho$ gives, since $(x+y)^*(x+y)\in J_+$ that
$$\lVert x+y\rVert^2 = \lVert x^*x+y^*y\rVert_J =  \sup_{\rho \in J^*_+, \rho(1) = 1} \rho((x+y)^*(x+y)) \le (\lVert x\rVert + \lVert y\rVert)^2.$$
\textbf{Step 6.} $\lVert\cdot\rVert$ satisfies axioms (i), (ii) and (iii) of a pre-C*-algebra.\\
Let $x,y \in A$. Since $\lVert\cdot\rVert$ extends $\lVert\cdot\rVert_J$ by the previous step, we have $\lVert x^*x\rVert = \lVert x^*x\rVert_J := \lVert x\rVert^2$, i.e.\ axiom (ii) holds.
Note that the map $R_y \colon J \to J$, $x \mapsto y^*xy$ is positive. Indeed, if $a \in J_+$ then $a = u^2$ for some $u\in J$, hence $R_y(a) = y^*u^2y = (uy)^*(uy) \in J_+$ by the last claim. Since $R_y$ is a positive operator, its operator norm is given by $\lVert R_y\rVert = \lVert R_y(1)\rVert_J = \lVert y^*y\rVert_J$. It follows that
$$\lVert xy\rVert^2 = \lVert y^*x^*xy^*\rVert_J = \lVert R_y(x^*x)\rVert_J \le \lVert R_y\rVert \lVert x^*x\rVert_J = \lVert x^*x\rVert_J \lVert y^*y\rVert_J = (\lVert x\rVert \lVert y\rVert)^2,$$
proving (i). Since $x^*x, y^*y \in J_+$ by step 5 and the JB-norm $\lVert\cdot\rVert_J$ is monotone on $J_+$, we obtain
\begin{equation*}
\lVert x^*x\rVert = \lVert x^*x\rVert_J \le \lVert x^*x + y^*y\rVert_J = \lVert x^*x + y^*y\rVert.
\end{equation*}
Thus axiom (iii) holds. We conclude that $\lVert\cdot\rVert$ makes $A$ into a pre-C*-algebra.\\
\textbf{Step 7.} In the complex case, $A$ is complete.\\
Let $x = a + ib \in A$ with $a,b\in J$. Since $a = \frac{1}{2}(x+x^*)$ and $b = \frac{1}{2i}(x-x^*)$ we have $$\max\{\lVert a\rVert_J, \lVert b\rVert_J\} \le \lVert a+ib\rVert \le \lVert a\rVert_J + \lVert b\rVert_J.$$
From these inequalities it is clear that $A$ is complete whenever $J$ is complete. 
\end{proof}

\begin{remark}\label{PthEg} (1) In contrast with the complex case, it is necessary to assume that $x^*x = 0$ implies $x = 0$, as the following example shows. Let $J = C_0(X)$ be an associative JB-algebra and let $A = J \times J$ be the real *-algebra having coordinatewise product and involution $(a,b)^* = (b, a)$. The diagonal embedding yields a Jordan isomorphism $J \stackrel{\sim}{\to} A_{sa} = \{(a,a)\colon a \in J\}$, so $A_{sa}$ admits a JB-norm. However, for $x := (a, 0)$ with $a\in J \setminus \{0\}$ we have $x^*x=(0, a)(a,0)=(0,0)$, which implies that every norm on $A$ violates axiom (ii) of a C*-algebra.\\
(2) In the real case, it need not hold that $A$ is complete. For example, let $I$ be the non-closed ideal of $C([0,1]; \R)$ consisting of those functions which vanish in a neighborhood of $0$. Consider the non-closed real *-subalgebra of $C([0,1];\C)$ given by
$$A := \{f \in C([0,1]; \C) : \Im f \in I\}.$$
Then $A$ is not complete, despite the completeness of its self-adjoint part $A_{sa} = C([0,1];\R)$.
\end{remark}

We now state a coordinatization theorem due to Hanche-Olsen, characterizing those JB-algebras which are the self-adjoint part of a complex C*-algebra.

\begin{theorem}[Hanche-Olsen]\label{HO}
Let $J$ be a unital JB-algebra and suppose $(J \otimes_\R M_2(\C))_{sa}$ is endowed with a Jordan product $\circ$ satisfying for all $x,y \in J$ and $a,b \in M_2(\C)_{sa}$:
\begin{align}
(x \otimes I_2) \circ (y \otimes I_2) &= (x \circ y) \otimes I_2, \label{HO1}\\
(1 \otimes a) \circ (1 \otimes b) &= 1 \otimes (a \circ b), \label{HO2}\\
(1 \otimes a) \circ (x \otimes I_2) &= x \otimes a, \label{HO3}\\
[L(1 \otimes a), L(x\otimes 1)] &= 0. \label{HO4}
\end{align}
Then there exists a unique associative product in $A := J \otimes_\R \C$ inducing the given Jordan products both in $J$ and in $(J \otimes_\R M_2(\C))_{sa}$. Moreover, $A$ with involution $(x \otimes \lambda)^* = x \otimes \overline{\lambda}$ is a C*-algebra for a norm extending $\lVert\cdot\rVert$ such that $A_{sa} = J$.
\end{theorem}
\begin{proof}
Except for the last sentence, this is shown in \cite[Thm. 4]{HaOlsTP}. The proof of the remaining part is not given in detail in \cite[Rmk. 11]{HaOlsTP}, so will be discussed here. Let us show that the conjugation on $A = J \otimes_\R \C$ is indeed an anti-homomorphism. For $x,y \in J$ we have
\begin{equation}
2\begin{pmatrix}
x & 0\\
0 & 0
\end{pmatrix}
\circ
\begin{pmatrix}
0 & y\\
y & 0
\end{pmatrix}
=
\begin{pmatrix}
0 & xy\\
yx & 0
\end{pmatrix}.
\end{equation}
Since the rightmost matrix is self-adjoint, it follows that $(xy)^*=yx=y^*x^*$. Since $A = J + iJ$, we obtain $(xy)^*=y^*x^*$ for all $x,y \in A$. The $C^*$-norm on $J$ is constructed in \Cref{norm}.
\end{proof}

We now state a coordinatization theorem due to Jacobson, and then show that for a JB-algebra the coordinate algebra is in fact a real C*-algebra. 

\begin{theorem}[Jacobson]\label{JacCood}
Let $J$ be a unital Jordan algebra over a field $\K$ of characteristic $\char \K \neq 2$. Let $n\ge 4$ and let $i\colon M_n(\K)_{sa} \to J$ be a unital Jordan homomorphism. Then there exists an associative $*$-algebra $D$ over $\K$ and a Jordan isomorphism $\Phi \colon M_n(D)_{sa} \to J$ extending $i$.
\end{theorem}
\begin{proof}
See \cite[Theorem 2.8.9]{HOSt84} or \cite[Thm. 17.1.1]{McC04}.
\end{proof}

\begin{corollary}
Let $J$ be a unital JB-algebra. Let $n\ge 4$ and let $i\colon M_n(\R)_{sa} \to J$ be a unital Jordan homomorphism. Then there exists a real C*-algebra $D$  and a JB-algebra isomorphism $\Phi \colon M_n(D)_{sa} \to J$ extending $i$.
\end{corollary}
\begin{proof}
After applying \Cref{JacCood} to obtain a unital Jordan isomorphism $\Phi\colon M_n(D)_{sa} \to J$, it remains to construct a C*-algebra norm on $D$ for which $\Phi$ is an isometry. By \Cref{OpCmm} the operator commutant $i(M_n(\R)_{sa})'$ is a JB-subalgebra of $J$. Its preimage under $\Phi$ is $M_n(\R)_{sa}' = D_{sa} \otimes I_n$. It follows that $\lVert d\rVert_{D_{sa}} := \lVert \Phi(d \otimes I_n)\rVert$ defines a JB-norm on $D_{sa}$.
Let $x\in D$; we claim that $x^*x = 0$ implies $x = 0$. Indeed, we have $(x \otimes E_{12} + x^* \otimes E_{21})^2 = xx^* \otimes E_{12} + x^*x \otimes E_{21}$ and the Jordan algebra $M_n(D)_{sa}$ contains no non-zero nilpotent elements, because the Jordan-isomorphic JB-algebra $J$ does not. According to \Cref{norm}, $D$ is a real pre-C*-algebra for a norm $\lVert \cdot \rVert$ extending $\lVert\cdot\rVert_{D_{sa}}$. We shall prove that $D$ is complete. Denote by $\hat{D}$ the real C*-algebra which is the completion of $(D, \lVert \cdot \rVert)$. Then $\Phi^{-1}\colon J \to M_n(\hat{D})_{sa}$ is an injective Jordan homomorphism between JB-algebras, so it is an isometry according to \cite[Prop. 1.35]{AS03}. Consequently, its image $M_n(D)_{sa} = \Phi^{-1}(J)$ is closed in $M_n(\hat{D})_{sa}$, which implies that $D = \hat{D}$. We conclude that $D$ is a real C*-algebra and that $\Phi\colon M_n(D)_{sa} \to J$ is an isometric isomorphism of JB-algebras.
\end{proof}

\section{JB-algebra structures on operator systems}
The main result of this section is \Cref{circdot} which asserts that a JB-algebra structure on the matrix space of an operator system is automatically compatible with the bimodule structure. The proof is based on \Cref{BicpThm} and the following lemma.

\begin{lemma}\label{CprBicp} Let $(E, *, \{ M_n(E)_+\}_n, e)$ be a real or complex operator system and let $n \in \Z_{\ge 1}$. Let $p$ and $q=1-p$ be complementary projections in $M_n(\K)$. Define $P,Q \colon M_n(E)_{sa} \to M_n(E)_{sa}$ by  $P(x) = p \cdot x \cdot p$ and $Q(x) = q \cdot x \cdot q$. Then $P$ and $Q$ are bicomplementary projections on the ordered normed space $M_n(E)_{sa}$.
\end{lemma}
\begin{proof}
Clearly $P$ and $Q$ are idempotent linear maps. The matrix properties of the order \eqref{matrOrd} and norm \eqref{mtrNorm} on $M_n(E)$ yield that $P$ and $Q$ are positive and contractive. To prove that $P$ and $Q$ are complementary projections, by symmetry it suffices to show that $\im_+ P = \ker_+ Q$.
Since $(Q \circ P)(x) = q \cdot (p \cdot x \cdot p) \cdot q = (qp) \cdot x \cdot (pq) = 0 \cdot x \cdot 0 = 0$ for $x \in M_n(E)_{sa}$, one has $Q\circ P = 0$, hence $\im_+ P \subset \ker_+ Q$. To show the converse, let $x \in \ker_+ Q$, i.e.\ $x\in M_n(E)_+$ and $q\cdot x\cdot q = 0$. Let $\phi\colon E \to \K$ be any state. Since states are completely positive, the induced map $\phi_n\colon M_n(E) \to M_n(\K)_n$ is positive, hence $\phi_n(x) \in M_n(\K)_+$ and $q\phi_n(x)q = 0$. Let $a \in M_n(\K)_{sa}$ with $a^2 = \phi_n(x)$. Then $(aq)^*aq=qa^2q=0$, hence $aq=0$ and $\phi_n(x)q=a^2q=0$. Similarly, $q\phi_n(x) = 0$. This implies that $\phi_n(x) = p\phi_n(x)p =\phi_n(p \cdot x \cdot p)$. Since the states on $E$ separate points, $x = p \cdot x \cdot p \in \im_+ P$, as desired.\\
\ind It remains to show that $P^*$ and $Q^*$ are also complementary. Again by symmetry we need only prove that $\im_+ P^* = \ker_+ Q^*$. Since $P \circ Q = 0$, we have $Q^* \circ P^* = 0$ and $\im_+ P^* \subset \ker_+ Q^*$. Now let $\phi \in \ker_+ Q^*$, so $\phi(M_n(E)_+) \subset \R_+$ and $\phi(q \cdot M_n(E)_{sa} \cdot q) = 0$.\\
\ind Let $y \in p \cdot M_n(E) \cdot q \oplus q \cdot M_n(E) \cdot p$ be self-adjoint. We will prove that $\phi(y) = 0$. By scaling, we may assume that $\lVert y\rVert \le 1$. For every $n \ge 1$, we claim that $\frac{1}{n}p + y + nq$ is positive. Indeed, since $\lVert y\rVert \le 1$ we have $p + y + q = I_n + y \ge 0$, hence $$\tfrac{1}{n}p + y + nq = (\tfrac{1}{\sqrt{n}} p + \sqrt{n}q) \cdot (p + y + q) \cdot (\tfrac{1}{\sqrt{n}}p + \sqrt{n} q) \ge 0.$$ Since $\phi(q) = 0$, we obtain $\frac{1}{n}\phi(p) + \phi(y) = \phi(\frac{1}{n}p + y + q)\ge 0$ for every $n\ge 1$, hence $\phi(y) \ge 0$. Applying the same reasoning to $-y$ gives $\phi(y) \le 0$, hence $\phi(y) = 0$.\\
\ind Now let $x \in M_n(E)_{sa}$ be arbitrary. By the previous paragraph, we find that
\begin{equation*}
\phi(x) = \phi(p \cdot x \cdot p) + \phi(p \cdot x \cdot q + q \cdot x \cdot p) + \phi(q \cdot x \cdot q) = \phi(p \cdot x \cdot p) = \phi(P(x)),    
\end{equation*} so $\phi \in \im P^*$ as desired. This shows that $\ker_+ Q^* \subset \im_+ P$, hence $\im_+ P^* = \ker_+ Q^*$. We conclude that $P$ and $Q$ are bicomplementary projections on $M_n(E)_{sa}$.
\end{proof}

\begin{theorem}[Automatic compatibility of JB-algebra and operator system structures]\label{circdot}
Let $(E, *, \{M_n(E)_+\}_n, e)$ be a real or complex operator system and let $n \in \Z_{\ge 1}$. Assume that the order unit space $(M_n(E)_{sa}, M_n(E)_+, I_n)$ is given a compatible structure of unital JB-algebra. Then:\\
(1) For all $a,b \in M_n(\K)_{sa}$ and $x \in M_n(E)_{sa}$ we have
\begin{equation}\label{Uap}
U_a(x) = a \cdot x \cdot a,
\end{equation}
and
\begin{equation}\label{Uabp}
\{a,x,b\} = \tfrac{1}{2}(a \cdot x \cdot b + b \cdot x \cdot a).
\end{equation}
Here, the left hand side is given by the quadratic resp.\ triple product in the JB-algebra $M_n(E)_{sa}$, while the right hand side is given by the $M_n(\K)$-bimodule structure on $M_n(E)$.\\
(2) The inclusion map $i\colon M_n(\K)_{sa} \to M_n(E)_{sa}$ is a unital Jordan homomorphism.\\
(3) The operator commutant of $M_n(\K)_{sa}$ is the JB-subalgebra $E_{sa} \otimes I_n := \{d \otimes I_n : d \in E_{sa}\}$, i.e.
\begin{equation}
M_n(\K)_{sa}' = E_{sa} \otimes I_n.
\end{equation}
\end{theorem}
\begin{proof}
(1)
Note that the members of \eqref{Uap} define quadratic maps $Q_i\colon M_n(\K)_{sa} \to \End_\R(M_n(E)_{sa})$, $i=1,2$, given by $Q_1(a) = U_a$ and $Q_2(a)(x) = a \cdot x \cdot a$, whose associated bilinear maps are given by \eqref{Uabp}. It therefore suffices to prove that $Q_1 = Q_2$.\\
\ind Let $p$ be a projection in $M_n(\K)_{sa}$, and let $q = 1-p$. As in \Cref{CprBicp}, the maps $P$ and $Q$ defined by $P(x) = p\cdot x\cdot p$ and $Q(x) = q \cdot x \cdot q$ are bicomplementary projections on $M_n(E)_{sa}$. \Cref{BicpThm} yields that $p = P(I_n)$ and $q = Q(I_n)$ are complementary projections in the JB-algebra $M_n(E)_{sa}$ such that $P = U_p$ and $Q = U_q$. We conclude that $Q_1(p) = U_p = P Q_2(p)$.\\
\ind Now let $a \in M_n(\K)_{sa}$ be arbitrary. By the spectral theorem we can write $a = \sum_{i=1}^n \lambda_i p_i$ where the $p_i$ are pairwise orthogonal projections and $\lambda_i \in \R$. For $k \in \{1,2\}$, polarization yields
$$Q_k(a) = \sum_{i=1}^n \lambda_i^2 Q_k(p_i) + \sum_{1\le i<j\le n} \lambda_i\lambda_j (Q_k(p_i + p_j) - Q_k(p_i) - Q_k(p_j)).$$
Since $p_i$, $p_j$ and $p_i + p_j$ are projections, it follows that $Q_1(a) = Q_2(a)$, as desired.

\ind (2) Set $x=I_n$ in \eqref{Uabp} to find $a \circ b = \{a, I_n, b\} = \frac{1}{2}(a \cdot I_n \cdot b + b \cdot I_n \cdot a) = \frac{1}{2}(ab + ba)$. This shows that the inclusion map is a Jordan homomorphism. It is unital by assumption.\\
\ind (3) Let $a \in M_n(E)_{sa}$ and let $S = \{s\in M_n(\K)_{sa}: s^2 = 1\}$ be the set of symmetries in $M_n(\K)_{sa}$. Since the JB-algebra $M_n(\K)_{sa}$ is generated by $S$, by \Cref{OpCmm} we have $M_n(\K)'_{sa} = S'$. If $s\in S$, then $a$ and $s$ operator commute if and only if $U_s(a) = a$ by \cite[equation (2.26)]{AS03}. Using \eqref{Uabp} we find
$$M_n(\K)'_{sa} = S' = \{a\in M_n(E)_{sa}: s \cdot a \cdot s = a\}.$$
For $d\in E_{sa}$ and $s\in S$ we have $s \cdot (d \otimes I_n) \cdot s = d \otimes s^2 = d \otimes I_n$, hence $E_{sa} \otimes I_n \subset S'$. The reverse inclusion follows upon considering symmetries having the form $s = 1-2E_{ii}$ and $s = 1 - E_{ii} - E_{jj} + E_{ij} + E_{ji}$, where $E_{ij}$ are the matrix units of $M_n(\K)$. By \Cref{OpCmm} we conclude that $E_{sa} \otimes I_n = M_n(\K)'_{sa}$ is a JB-subalgebra of $M_n(E)_{sa}$.
\end{proof}

\section{Proof of main results}
This sections contains the demonstrations of the theorems stated in the introduction. First, we will prove \Cref{MAINTHM2} in the following more precise form. Recall that a $*$-\emph{isomorphism} between $*$-vector spaces is a self-adjoint linear isomorphism.

\begin{theorem}\label{MAINJB} Let $(E, *, \{M_n(E)_+\}_n, e)$ be a (real or complex) complete unital operator system. Let $S$ be the set of integers $n \ge 1$ for which the order unit space $(M_n(E)_{sa}, M_n(E)_+^{\circ}, I_n)$ admits a compatible JB-algebra structure (necessarily unique by \Cref{Kap}).\\
\ind (1) If $1\le m < n$ and $n \in S$ then $m \in S$.\\
Suppose $S$ contains an integer $n$ with $n\ge 2$ if $E$ is complex and $n\ge 4$ if $E$ is real.\\
\ind (2) Then there exists a unital C*-algebra $A$ and a unital *-isomorphism $\phi\colon A \to E$ such that $\phi_m \colon M_m(A)_{sa} \to M_m(E)_{sa}$ is a JB-algebra isomorphism for all $m \in S$.
\end{theorem}


\begin{proof}[Proof~of~part~(1)]
Since $n \in S$, the order unit space $(M_n(E)_{sa}, M_n(E)_+, I_n)$ admits a unique compatible structure of unital JB-algebra. Let $P\colon M_n(\C)_{sa} \to M_m(\C)_{sa}$ be the compression onto the upper left $m \times m$-block. By Theorems \ref{BicpThm} and \ref{CprBicp} we have that $P(M_m(E))_{sa}$ is a JB-algebra with unit element $P(I_m)$. This unital JB-algebra is isomorphic to the order unit space $M_m(E)_{sa},$ which therefore has a Thompson symmetric cone. We conclude that $m \in S$.
\end{proof}

\begin{proof}[Proof~of~part~(2),~complex~case]
We are going to apply \Cref{HO}. By part (a), from $2\le n \in S$ it follows that $2\in S$. Hence, $(M_2(E)_{sa}, M_2(E)_+, I_2)$ admits a (unique) compatible JB-algebra structure. By \Cref{circdot}(3), $E_{sa} \otimes I_2$ is a JB-subalgebra of $M_2(E)_{sa}$. Using the isomorphism $E_{sa} \cong E_{sa} \otimes I_2$, we give $E_{sa}$ a JB-algebra structure such that \eqref{HO1} is satisfied. It follows from \eqref{Uabp} that \eqref{HO2} and \eqref{HO3} are satisfied. Finally, \Cref{circdot}(3) yields that \eqref{HO4} is satisfied. Thus we obtain from \Cref{HO} that there uniquely exists a unital C*-algebra $A$ and a unital \mbox{*-iso}morphism $\phi\colon A \to E$ such that $\phi_m \colon M_m(A)_{sa} \to M_m(E)_{sa}$ is an isomorphism of JB-algebras for $m \in \{1,2\}$.\\
\ind It remains to show that $\phi_m\colon M_m(A)_{sa} \to M_m(E)_{sa}$ is a JB-algebra isomorphism for all $m \in S$. Note that $M_m(A)_{sa} = A_{sa} \otimes_\R M_n(\C)_{sa}$. Let $x,y \in M_n(A)_{sa}$. Since any element of $E_{sa} \otimes_\R M_n(\C)_{sa}$ is a sum of pure tensors, and any element of $M_n(\C)_{sa}$ is a real linear combination of rank-$1$ projections, we have $x = \sum_i c_i \otimes q_i$ and $y = \sum_j d_j \otimes r_j$ for suitable elements $c_i, d_j \in E_{sa}$ and rank-$1$ projections $q_i, r_j \in M_n(\C)_{sa}$. By bilinearity of the Jordan products in $M_n(A)_{sa}$ and in $M_n(E)_{sa}$, it therefore suffices to show for $c,d\in E_{sa}$ and rank-$1$ projections $q,r\in M_n(\R)_{sa}$ that 
\begin{equation}\label{prod}
\phi_m((c \otimes q) \circ (d \otimes r)) = \phi_m(c \otimes q) \circ \phi_m(d \otimes r).
\end{equation}
\ind Choose an isometry $u \in M_{2,n}(\C)$ such that the projection $p:= uu^*$ dominates $q$ and $r$. Denote the compression map attached to $u$ by $C\colon M_m(A) \to M_2(A)$, $C(z) = u \cdot z \cdot u$, and similarly for $E$. Note that $p\cdot M_m(E)_{sa} \cdot p$ is a JB-subalgebra of $M_m(E)_{sa}$ by \Cref{BicpThm} and \Cref{CprBicp}, and likewise $p \cdot M_m(A)_{sa} \cdot p$ is a JB-subalgebra of $M_m(A)_{sa}$. We obtain a commutative diagram
\begin{center}
\begin{tikzcd}
p\cdot M_m(D)_{sa} \cdot p \arrow[d, "C"'] \arrow[rr, "\phi_m"] &  & p\cdot M_m(E)_{sa} \cdot p \arrow[d, "C"] \\
M_2(D)_{sa} \arrow[rr, "\phi_2"]                                &  & M_2(E)_{sa},    
\end{tikzcd}
\end{center}
in which the vertical arrows are surjective unital isometries, hence Jordan isomorphisms. The bottom map is a Jordan isomorphism as well, hence so is the top map by commutativity. Since $c \otimes q, d \otimes r \in p \cdot M_m(A)_{sa} \cdot p$, we conclude that \eqref{prod} is satisfied.
\end{proof}


\begin{proof}[Proof~of~part~(2),~real~case]
Choose $n\in S$ with $n\ge 4$ (by part (a), one may take $n=4$). We apply \Cref{JacCood} to the JB-algebra $J := M_n(E)_{sa}$ and the inclusion map $i\colon M_n(\R)_{sa} \to M_n(E)_{sa}$ of \Cref{circdot}(2), to obtain a real C*-algebra $A$ and a unital JB-isomorphism $\Phi\colon M_n(A)_{sa} \to J$. Put $H := M_n(A)_{sa}$. We will define a unital *-isomorphism $\phi\colon A \to E$ such that $\Phi = \phi_n|_H$.\\
\ind Denote the matrix units of $M_n(\R)$ by $E_{ij}$.
Note that $E_{11}, \ldots, E_{nn}$ is a complete system of orthogonal idempotents in both $H$ and $J$. We denote the corresponding Peirce eigenspaces by $H_{ij}$ and $J_{ij}$. Since $\Phi$ extends $i$, it preserves the Peirce decomposition relative to $\{E_{ii}\}$, i.e.\ $\Phi(H_{ij}) = J_{ij}$. Using \Cref{circdot}(1) we find for $1\le i\neq j\le n$ that
$$J_{ii} = \{y \otimes E_{ii}: y \in E_{sa}\}, \quad J_{ij} = \{y \otimes E_{ij} + y^* \otimes E_{ji}: y \in E\}.$$
Similarly,
$$H_{ii} = \{x \otimes E_{ii}: x \in A_{sa}\}, \quad H_{ij} = \{x \otimes E_{ij} + x^* \otimes E_{ji}: x \in A.\}$$
Since $\Phi$ is the identity on $M_n(\R)_{sa}$, it preserves the Peirce decomposition relative to $\{E_{ii}\}$. In other words, $\Phi$ restricts for all $1\le i,j\le n$ to a bijection $$\Phi|_{H_{ij}}\colon H_{ij} \stackrel{\sim}{\to} J_{ij}.$$
Therefore, there exists a unique bijection $\phi\colon A \stackrel{\sim}{\to} E$ such that for each $x\in A$ it holds that
\begin{equation}\label{phidef}
\Phi(x \otimes E_{12} + x^* \otimes E_{21}) = \phi(x) \otimes E_{12} + \phi(x)^* \otimes E_{21}.
\end{equation}
We are going to check that $\phi$ is a unital *-isomorphism such that $\Phi = \phi_n|_H$. To this end, we shall first record some identities featuring the elements of $M_n(\R)_{sa}$ given for $1\le i\neq j\le n$ by
\begin{equation}
h_{ij} := E_{ij} + E_{ji}, \quad
s_{ij} := I_n - E_{ii} - E_{jj} + h_{ij}.  
\end{equation} Let $\pi$ be a permutation of $\{1,2,\ldots,n\}$ which is a transposition interchanging $k$ and $l$, say. By repeated use of \Cref{circdot}(1), one checks that for all $1\le i\neq j \le n$ we have
\begin{align}
U_{s_{kl}}(y \otimes E_{ii}) &= y \otimes E_{\pi(i)\pi(i)}, \label{Usklii}\\
U_{s_{kl}}(y \otimes E_{ij} + y^* \otimes E_{ji}) &= y \otimes E_{\pi(i)\pi(j)} + y^* \otimes E_{\pi(j)\pi(i)}. \label{Usklij}
\end{align}
Moreover, for $y\in E_{sa}$ we have
\begin{equation}\label{brh}
2h_{12} \circ (y \otimes E_{11}) = y \otimes E_{12} + y \otimes E_{21}.
\end{equation}
The analogous equations hold in $M_n(D)_{sa}$, directly from the definition of its Jordan product.\\
\ind It follows from \eqref{phidef} with $x=1$ and $\Phi(h_{12}) = h_{12}$ that $\phi$ is unital. Moreover, \eqref{Usklij} with $i=k=1$ and $j=l=2$ gives that
$$U_{s_{12}}(x \otimes E_{12} + x^* \otimes E_{21}) = x^* \otimes E_{12} + x \otimes E_{21}$$
and similarly for $y \in A$. 
Since $\Phi \circ U_{s_{12}} = U_{s_{12}} \circ \Phi$, it follows that $\phi$ is $*$-preserving. From \eqref{brh} and \eqref{phidef} it follows that $\Phi(x \otimes E_{11}) = \phi(x) \otimes E_{11}$ for $x\in A_{sa}$. This establishes that $\Phi$ and $\phi_n$ agree on $A_{11} + A_{12}$. Using \eqref{Usklii} and \eqref{Usklij} with $i=k=1$, $j=2$ and $l\ge 2$ it follows that $\Phi$ and $\phi_n$ agree on $A_{ll} + A_{2l}$. A similar application of \eqref{Usklij} with $\pi(2)=k$ shows that $\Phi$ and $\phi_n$ agree on $A_{kl}$ in the remaining case that $k \neq l$ and $l \ge 2$. We conclude that $\Phi = \phi_n|_H$.




\ind To finish the proof of part (b), let us show that $\phi_m\colon M_m(A)_{sa} \to M_m(E)_{sa}$ is a unital order isomorphism whenever $m \in S$. This is clear if $m < n$, so assume $m > n$. Invoking \Cref{MRJB}, we endow $(M_m(E)_{sa}, M_m(E)_+, I_m)$ with a compatible structure of JB-algebra and claim that then $\phi_m$ is a Jordan isomorphism.\\ \ind Let $x,y \in M_m(E)_{sa}$; it is to be shown that $\phi_m(x \circ y) = \phi_m(x) \circ \phi_m(y)$. By bilinearity of the product, we can assume $x = c \otimes E_{ij} + c^* \otimes E_{ji}$ and $y = d \otimes E_{kl} + d^* \otimes E_{lk}$ for $c,d\in A$ and integers $1\le i,j,k,l\le n$. Let $u \in M_{n,m}(\R)_{sa}$ be an isometry such that the projection $p \in M_n(\R)_{sa}$ dominates $E_{ii}$, $E_{jj}$, $E_{kk}$, $E_{ll}$. Since $x,y \in p \cdot M_m(D)_{sa} \cdot p$, the argument given in the last paragraph of the complex case shows that $\phi(x\circ y)=\phi(x) \circ \phi(y)$. We conclude that $\phi_m\colon M_m(A)_{sa} \to M_m(E)_{sa}$ is a JB-algebra isomorphism. The proof of the theorem is complete.
\end{proof}

We are now in the position to prove the main result of this article. Part (2) of the following theorem was stated as \Cref{MAINTHM} in the introduction.

\begin{theorem}\label{MAINTHMGR} Let $(E, *, \{M_n(E)_+\}_n, e)$ be a (real or complex) complete unital operator system. Let $m \in \Z$ be an integer such that $m\ge 2$ if $E$ is complex and $m\ge 4$ if $E$ is real. Then:\\
(1) $E$ is unitally $m$-order isomorphic to a C*-algebra if and only if $M_m(E)_+^{\circ}$ admits a gauge-reversing bijection.\\
(2) $E$ is unitally completely order isomorphic to a C*-algebra if and only if  $M_n(E)^\circ_+$ admits a gauge-reversing bijection for arbitrarily large integers $n\ge 1$.
\end{theorem}
\begin{proof}
If $A$ is a unital C*-algebra, then for every $n\ge 1$ also $M_n(A)$ is a unital C*-algebra whose inversion map defines a gauge-reversing bijection on $M_n(A)_+^{\circ}$. Now, if $E$ is unitally $m$-order isomorphic to $A$, then for every $n\le m$ the order unit spaces $M_n(E)_{sa}$ and $M_n(A)_{sa}$ 
are isomorphic, and hence there exists a gauge-reversing bijection on $M_n(E)_+^{\circ}$.\\
\ind Conversely, let $S$ be the set of integers $n\ge 1$ for which $M_n(E)_+^{\circ}$ admits a gauge-reversing bijection. The order-theoretic characterization of JB-algebras in \Cref{MRJB} implies that there exists a compatible JB-algebra structure on $(M_n(E), M_n(E)_+, I_n)$ for every $n\in S$. The proof is finished by an application of \Cref{MAINJB}.
\end{proof}

\begin{example} The following example demonstrates that, in contrast to the complex case, the real case of part (1) of \Cref{MAINTHMGR} does not hold if one merely assumes $m\ge 2$.\\
\ind Let $k$ be a positive integer with $k \not\in \{1, 2, 4\}$. There exist pairwise anti-commuting skew-adjoint unitaries $s_1, \ldots, s_{k-1} \in B(H)$ on a complex Hilbert space $H$ of sufficiently large dimension. Let $E \subset B(H)$ be the unital real operator system given by the real linear span of $\Id_H$ and $s_1, \ldots, s_{k-1}$. Then $M_2(E)_{sa}$ is a JB-subalgebra of $B(H)_{sa}$, hence $M_2(E)_+^{\circ}$ admits a gauge-reversing bijection. We claim that $E$ cannot be unitally $2$-order isomorphic to a real C*-algebra $A$. In fact, if $E$ and $A$ were *-isomorphic, then $A_{sa} \cong E_{sa} = \R$. For $y\in A$ we have $yy^* = \lVert y\rVert^2 = y^*y$, from where $A$ is a normed associative division algebra over $\R$. But then by Mazur's theorem \cite{Mazur} this implies that $A$ is isomorphic to $\R$, $\C$ or $\H$, hence $E$ has dimension $1$, $2$ or $4$. Since $\dim E = k \not\in \{1, 2, 4\}$, we arrive at a contradiction.\\
\ind The author does not know of a counterexample with $m=3$.
\end{example}

\Cref{MAINJB} may provide a new perspective on the Choi--Effros theorem \cite[Thm. 3.1]{ChEff77} stating that the range of a unital completely positive projection on a unital complex C*-algebra is unitally completely order isomorphic to a unital complex C*-algebra. The analogous result for real unital C*-algebras was established by Ruan in \cite[Section 4]{Ruan01} for real $B(H)$ and in the general case by Blecher in \cite[p. 5]{Ble24}. Neal--Russo proved in \cite[Cor. 5.6]{NlRs02} that the range of a unital $2$-positive projection on a unital complex C*-algebra is unitally $2$-order isomorphic to a unital complex C*-algebra. All their results may be deduced in one fell swoop from the corresponding result for JB-algebras due to Effros--Størmer (cited as \Cref{EffSt}) by an application of \Cref{MAINJB}. This approach also allows to treat the case of a unital $n$-positive projection on a unital real C*-algebra for $n\ge 4$.

\begin{theorem}\label{UPJB} Let $A$ be a real/complex unital C*-algebra. Let $n$ be an integer such that $n\ge 4$ if $A$ is real and $n\ge 2$ if $A$ is complex. Suppose $P\colon A \to A$ is a unital $n$-positive (resp.\ completely positive) idempotent linear map. Then there exists a unital $n$-order isomorphism (resp.\ complete order isomorphism) from the range $P(A)$ onto a unital real/complex C*-algebra.
\end{theorem}
\begin{proof}
Let $m\ge 1$ be an integer such that $P$ is $m$-positive. We consider the unital JB-algebra $J := M_m(A)_{sa}$. Then $P_m\colon J \to J$ is a unital positive projection, hence $P_m(J) = M_m(P(A))_{sa}$ is unitally order isomorphic to a unital JB-algebra by \Cref{EffSt}. The conclusion now follows from \Cref{MAINJB}(2). 
\end{proof}

\section{Acknowledgements}
The author would like to thank Mark Roelands for his valuable comments on an earlier version of this article. Furthermore, the author is grateful to David Blecher for pointing out several relevant references.


\bibliography{references}

@Book{AS01,
 Author = {Alfsen, E. M. and Shultz, F. W.},
 Title = {State spaces of operator algebras. {Basic} theory, orientations, and {{\(C^*\)}}-products},
 ISBN = {0-8176-3890-3},
 Year = {2001},
 Publisher = {Birkh{\"a}user, Boston, MA},
 zbMATH = {1614150},
 Zbl = {0983.46047}
}

@article {ElMCRP,
    AUTHOR = {El Amin, K.  and Morales Campoy, A.   and 
              Rodr\'iguez Palacios, A. },
     TITLE = {A holomorphic characterization of {$C^*$}- and
              {$JB^*$}-algebras},
   JOURNAL = {Manuscripta Math.},
  FJOURNAL = {Manuscripta Mathematica},
    VOLUME = {104},
      YEAR = {2001},
    NUMBER = {4},
     PAGES = {467--478},
      ISSN = {0025-2611,1432-1785},
   MRCLASS = {46L70 (46L05)},
  MRNUMBER = {1836108},
MRREVIEWER = {Cho\ Ho\ Chu},
       DOI = {10.1007/s002290170021},
       URL = {https://doi.org/10.1007/s002290170021},
}

@article {Sak56,
    AUTHOR = {Sakai, Sh\^oichir\^o},
     TITLE = {The absolute value of {$W^*$}-algebras of finite type},
   JOURNAL = {Tohoku Math. J. (2)},
  FJOURNAL = {The Tohoku Mathematical Journal. Second Series},
    VOLUME = {8},
      YEAR = {1956},
     PAGES = {70--85},
      ISSN = {0040-8735,2186-585X},
   MRCLASS = {46.1X},
  MRNUMBER = {81453},
MRREVIEWER = {I.\ E.\ Segal},
       DOI = {10.2748/tmj/1178245010},
       URL = {https://doi.org/10.2748/tmj/1178245010},
}

@article {Sher56,
    AUTHOR = {Sherman, S.},
     TITLE = {On {S}egal's postulates for general quantum mechanics},
   JOURNAL = {Ann. of Math. (2)},
  FJOURNAL = {Annals of Mathematics. Second Series},
    VOLUME = {64},
      YEAR = {1956},
     PAGES = {593--601},
      ISSN = {0003-486X},
   MRCLASS = {81.0X},
  MRNUMBER = {82887},
MRREVIEWER = {I.\ E.\ Segal},
       DOI = {10.2307/1969605},
       URL = {https://doi.org/10.2307/1969605},
}

@article {Kad51,
    AUTHOR = {Kadison, Richard V.},
     TITLE = {A representation theory for commutative topological algebra},
   JOURNAL = {Mem. Amer. Math. Soc.},
  FJOURNAL = {Memoirs of the American Mathematical Society},
    VOLUME = {7},
      YEAR = {1951},
     PAGES = {39},
      ISSN = {0065-9266,1947-6221},
   MRCLASS = {46.3X},
  MRNUMBER = {44040},
MRREVIEWER = {L.\ Nachbin},
}

@article{vdWopCm,
title = {Commutativity in {J}ordan operator algebras},
author = {van de Wetering, J.},
fjournal = {Journal of Pure and Applied Algebra},
 journal = {J. Pure Appl. Algebra},
volume = {224},
number = {11},
pages = {106407},
year = {2020},
}

@book{McC04,
 author = {McCrimmon, K.},
 title = {A taste of {Jordan} algebras},
 issn = {0172-5939},
 isbn = {0-387-95447-3},
 year = {2004},
 publisher =  {Springer, New York, NY},
 
 doi = {10.1007/b97489},
 zbMATH = {2043992},
 Zbl = {1044.17001}
}

@article{Connes74,
 author = {Connes, A.},
 title = {Caract{\'e}risation des espaces vectoriels ordonn{\'e}s sous-jacents aux alg{\`e}bres de {von} {Neumann}},
 fjournal = {Annales de l'Institut Fourier},
 journal = {Ann. Inst. Fourier (Grenoble)},
 issn = {0373-0956},
 volume = {24},
 number = {4},
 pages = {121--155},
 year = {1974},
 
 doi = {10.5802/aif.534},
 url = {https://eudml.org/doc/74194},
 zbMATH = {3450927},
 Zbl = {0287.46078}
}

@book{SakCW,
 author = {Sakai, S.},
 title = {{{\(C^ *\)}}-algebras and {{\(W^ *\)}}-algebras},
 fseries = {Ergebnisse der Mathematik und ihrer Grenzgebiete},
 series = {Ergeb. Math. Grenzgeb.},
 volume = {60},
 year = {1971},
 publisher = {Springer-Verlag, Berlin},
 language = {English},
 zbMATH = {3348793},
 Zbl = {0219.46042}
}

@article{UpmDer,
 ISSN = {00255521, 19031807},
 URL = {http://www.jstor.org/stable/24491372},
 author = {Harald Upmeier},
 journal = {Math. Scand.},
 number = {2},
 pages = {251--264},
 publisher = {Aarhus Univ., Aarhus},
 title = {Derivations of {J}ordan algebras*-ALGEBRAS},
 urldate = {2026-04-07},
 volume = {46},
 year = {1980}
}

@Book{AS03,
 Author = {Alfsen, E. M. and Shultz, F. W.},
 Title = {Geometry of state spaces of operator algebras},
 ISBN = {0-8176-4319-2},
 Year = {2003},
 Publisher = {Birkh{\"a}user, Boston, MA},
 zbMATH = {1885140},
 Zbl = {1042.46001}
}

@misc{LRW25,
      title={Infinite dimensional symmetric cones and gauge-reversing maps}, 
      author={B. Lemmens and M. Roelands and M. Wortel},
      year={Preprint, 2025},
      eprint={arXiv:2504.12487},
      archivePrefix={arXiv},
      primaryClass={math.FA},
      url={https://arxiv.org/abs/2504.12487}, 
}

@book{HOSt84,
 author = {Hanche-Olsen, H. and St{\o}rmer, E.},
 title = {Jordan operator algebras},
 fseries = {Monographs and Studies in Mathematics},
 series = {Monogr. Stud. Math.},
 volume = {21},
 year = {1984},
 publisher = {Pitman, Boston, MA},
 zbMATH = {3893796},
 Zbl = {0561.46031}
}

@article {ASchar78,
    AUTHOR = {Alfsen, E. M. and Shultz, F. W.},
     TITLE = {State spaces of {J}ordan algebras},
   JOURNAL = {Acta Math.},
  FJOURNAL = {Acta Mathematica},
    VOLUME = {140},
      YEAR = {1978},
     PAGES = {155--190},
      ISSN = {0001-5962,1871-2509},
   MRCLASS = {17C99 (46K10)},
  MRNUMBER = {472949},
MRREVIEWER = {Gerhard\ Janssen},
       DOI = {10.1007/BF02392307},
       URL = {https://doi.org/10.1007/BF02392307},
}

@incollection {Ros16,
    AUTHOR = {Rosenberg, Jonathan},
     TITLE = {Structure and applications of real {$C^*$}-algebras},
 BOOKTITLE = {Operator algebras and their applications},
    SERIES = {Contemp. Math.},
    VOLUME = {671},
     PAGES = {235--258},
 PUBLISHER = {American Mathematical Society, Providence, RI},
      YEAR = {2016},
      ISBN = {978-1-4704-1948-6},
   MRCLASS = {46L35 (19K35 19L50 19L64)},
  MRNUMBER = {3546688},
       DOI = {10.1090/conm/671/13513},
       URL = {https://doi.org/10.1090/conm/671/13513},
}

@book {Gdl82,
    AUTHOR = {Goodearl, K. R.},
     TITLE = {Notes on real and complex {$C\sp{\ast} $}-algebras},
    SERIES = {Shiva Mathematics Series},
    VOLUME = {5},
 PUBLISHER = {Shiva Publishing Ltd., Nantwich},
      YEAR = {1982},
     PAGES = {iv+211},
      ISBN = {0-906812-16-X},
   MRCLASS = {46Lxx (16-01 16A48 46-01)},
  MRNUMBER = {677280},
MRREVIEWER = {G.\ A.\ Elliott},
}

@book {Schr93,
    AUTHOR = {Schr\"oder, Herbert},
     TITLE = {{$K$}-theory for real {$C^*$}-algebras and applications},
    SERIES = {Pitman Research Notes in Mathematics Series},
    VOLUME = {290},
 PUBLISHER = {Longman Scientific \& Technical, Harlow; copublished in the
              United States with John Wiley \& Sons, Inc., New York},
      YEAR = {1993},
     PAGES = {xiv+162},
      ISBN = {0-582-21929-9},
   MRCLASS = {19K35 (19K56 19L64 46L80 55N15 57R30 58G12)},
  MRNUMBER = {1267059},
MRREVIEWER = {Jonathan\ M.\ Rosenberg},
}

@book{Paucboa,
 author = {Paulsen, Vern},
 title = {Completely bounded maps and operator algebras},
 fseries = {Cambridge Studies in Advanced Mathematics},
 series = {Camb. Stud. Adv. Math.},
 volume = {78},
 isbn = {0-521-81669-6},
 year = {2002},
 publisher = {Cambridge: Cambridge University Press},
 language = {English},
 zbMATH = {1849957},
 Zbl = {1029.47003}
}

@article{BleTep,
 author = {Blecher, David P. and Tepsan, Worawit},
 title = {Real operator algebras and real positive maps},
 fjournal = {Integral Equations and Operator Theory},
 journal = {Integral Equations Operator Theory},
 issn = {0378-620X},
 volume = {93},
 number = {5},
 pages = {33},
 year = {2021},
 language = {English},
}

@article{EffSto79,
 author = {Effros, E.G. and St{\o}rmer, E.},
 title = {Positive projections and {Jordan} structure in operator algebras},
 fjournal = {Mathematica Scandinavica},
 journal = {Math. Scand.},
 issn = {0025-5521},
 volume = {45},
 pages = {127--138},
 year = {1979},
 language = {English},
 doi = {10.7146/math.scand.a-11830},
 url = {https://eudml.org/doc/166668},
 zbMATH = {3712627},
 Zbl = {0455.46059}
}

@article {Ozawa,
    AUTHOR = {Ozawa, Narutaka},
     TITLE = {About the {C}onnes embedding conjecture},
   JOURNAL = {Jpn. J. Math.},
  FJOURNAL = {Japanese Journal of Mathematics},
    VOLUME = {8},
      YEAR = {2013},
    NUMBER = {1},
     PAGES = {147--183},
      ISSN = {0289-2316,1861-3624},
   MRCLASS = {46-02 (16W80 46Lxx)},
  MRNUMBER = {3067294},
       DOI = {10.1007/s11537-013-1280-5},
       URL = {https://doi.org/10.1007/s11537-013-1280-5},
}

@article{BleRus,
title = {Real operator systems},
journal = {J. Funct. Anal.},
volume = {290},
number = {2},
pages = {111-226},
year = {2026},
issn = {0022-1236},
author = {David P. Blecher and Travis B. Russell},
}

@article {ChuGeom,
    AUTHOR = {Chu, Cho-Ho},
     TITLE = {A geometric characterisation of real {C}*-algebras},
   JOURNAL = {Sci. China Math.},
  FJOURNAL = {Science China. Mathematics},
    VOLUME = {66},
      YEAR = {2023},
    NUMBER = {10},
     PAGES = {2277--2292},
      ISSN = {1674-7283,1869-1862},
   MRCLASS = {46L05 (17C65 22E65 58B20)},
  MRNUMBER = {4646971},
MRREVIEWER = {Junping\ Liu},
       DOI = {10.1007/s11425-022-2041-8},
       URL = {https://doi.org/10.1007/s11425-022-2041-8},
}

@incollection {UpmHoloC,
    AUTHOR = {Upmeier, Harald},
     TITLE = {A holomorphic characterization of {C}*-algebras},
 BOOKTITLE = {Functional analysis, holomorphy and approximation theory, {II}
              ({R}io de {J}aneiro, 1981)},
    SERIES = {North-Holland Math. Stud.},
    VOLUME = {86},
     PAGES = {427--467},
 PUBLISHER = {North-Holland, Amsterdam},
      YEAR = {1984},
      ISBN = {0-444-86845-3},
   MRCLASS = {46L05 (17C65 32M10 46K05)},
  MRNUMBER = {771340},
MRREVIEWER = {Autorreferat},
       DOI = {10.1016/S0304-0208(08)70840-5},
       URL = {https://doi.org/10.1016/S0304-0208(08)70840-5},
}

@article {Ble24,
    AUTHOR = {Blecher, David P.},
     TITLE = {Real operator spaces and operator algebras},
   JOURNAL = {Studia Math.},
  FJOURNAL = {Studia Mathematica},
    VOLUME = {275},
      YEAR = {2024},
    NUMBER = {1},
     PAGES = {1--40},
      ISSN = {0039-3223,1730-6337},
   MRCLASS = {46L07 (46H25 46L08 46M10 47L05 47L25 47L30 47L75)},
  MRNUMBER = {4746864},
MRREVIEWER = {Serap\ \"Oztop},
       DOI = {10.4064/sm230329-23-12},
       URL = {https://doi.org/10.4064/sm230329-23-12},
}

@article{Ruan01,
 author = {Ruan, Zhong Jin},
 title = {On real operator spaces},
 fjournal = {Acta Mathematica Sinica. English Series},
 journal = {Acta Math. Sin., Engl. Ser.},
 issn = {1439-8516},
 volume = {19},
 number = {3},
 pages = {485--496},
 year = {2003},
 language = {English},
}

@article{Mazur,
 author = {Mazur, Stanislaw},
 title = {Sur les anneaux lin{\'e}aires},
 journal = {C. R. Math. Acad. Sci. Paris},
 issn = {0001-4036},
 volume = {207},
 pages = {1025--1027},
 year = {1938},
 language = {French},

}

@article{ChEff77,
 author = {Choi, M.-D. and Effros, E.G.},
 title = {Injectivity and operator spaces},
 journal = {J. Funct. Anal.},
 issn = {0022-1236},
 volume = {24},
 pages = {156--209},
 year = {1977},
 language = {English},
 doi = {10.1016/0022-1236(77)90052-0},
 zbMATH = {3532026},
 Zbl = {0341.46049}
}

@article {NlRs02,
    AUTHOR = {Neal, Matthew and Russo, Bernard},
     TITLE = {Operator space characterizations of {$C^*$}-algebras and
              ternary rings},
   JOURNAL = {Pacific J. Math.},
  FJOURNAL = {Pacific Journal of Mathematics},
    VOLUME = {209},
      YEAR = {2003},
    NUMBER = {2},
     PAGES = {339--364},
      ISSN = {0030-8730,1945-5844},
   MRCLASS = {46L05 (46L70 46T05)},
  MRNUMBER = {1978376},
MRREVIEWER = {Richard\ M.\ Timoney},
       DOI = {10.2140/pjm.2003.209.339},
       URL = {https://doi.org/10.2140/pjm.2003.209.339},
}

@inproceedings{HaOlsTP,
    author = {Hanche-Olsen, H.},
    title = {{JB}-algebras with tensor product are {C}*-algebras},
    booktitle = {Operator Algebras and their Connections with Topology and Ergodic Theory},
    year = {1985}
}

@misc{OCJB,
      title={An order-theoretic characterization of {J}{B}-algebras}, 
      author={Mark Roelands and Samuel Tiersma},
      year={Preprint, 2025},
      eprint={arXiv:2507.09526},
      url={https://arxiv.org/abs/2507.09526}, 
}

\end{document}